\documentclass{article}
\usepackage[greek, english]{babel}
\usepackage{amsmath}
\usepackage{amssymb}
\usepackage{eurosym}
\usepackage{epsf}
\usepackage{epsfig}
\usepackage{tikz}
\usetikzlibrary{arrows}
\usepackage{url}

\def\ad{\mathop{\rm ad}\nolimits}
\def\im{\mathop{\rm im}\nolimits}

\newfont{\gothic}{eufm10}

   \def\C{{\CC}} \def\CC{{\mathbb{C}}}
                   \def\R{{\RR}} 
        \def\N{{\mathbb{N}}}

        \newtheorem{theorem}{Theorem}[section]

\newtheorem{proposition}[theorem]{Proposition}

\newtheorem{definition1}[theorem]{Definition}
\newenvironment{definition}{\begin{definition1}\rm}{\hfill $\triangle$\end{definition1}}
\newenvironment{proof}{\addvspace\baselineskip\noindent{\it
Proof:}}{\hspace*{\fill}         $\Box$\par\addvspace\baselineskip}
\newenvironment{proofof}[1]{\addvspace\baselineskip\noindent{\it Proof
\bf (of #1)}:}{\hspace*{\fill} $\Box$\par\addvspace\baselineskip}

\newtheorem{remark1}[theorem]{Remark}
\newenvironment{remark}{\begin{remark1}\rm}{\hfill $\triangle$\end{remark1}}

\newtheorem{example1}[theorem]{Example}

\def\barray{\begin{eqnarray*}}             \def\earray{\end{eqnarray*}}
\def\beq{\begin{equation}} \def\eeq{\end{equation}}

\makeatletter \title{Amplified Hopf bifurcations in feed-forward networks}
\author{Bob Rink\thanks{Department of Mathematics, VU University Amsterdam, The Netherlands, 
{\tt b.w.rink@vu.nl}.} \ and Jan Sanders\thanks{Department of Mathematics, VU University Amsterdam, 
The Netherlands, {\tt jan.sanders.a@gmail.com}.}
}
\begin{document}  \hyphenation{boun-da-ry mo-no-dro-my sin-gu-la-ri-ty ma-ni-fold ma-ni-folds 
re-fe-rence se-cond se-ve-ral dia-go-na-lised con-ti-nuous thres-hold re-sul-ting fi-nite-di-men-sio-nal 
ap-proxi-ma-tion pro-per-ties ri-go-rous mo-dels mo-no-to-ni-ci-ty pe-ri-o-di-ci-ties mi-ni-mi-zer 
mi-ni-mi-zers know-ledge ap-proxi-mate pro-per-ty poin-ting eigen-values}
 
\newcommand{\X}{\mathbb{X}}
\def\lieg{\mathfrak{g}}
\def\gl{\mathfrak{gl}}
\def\ad{\mathrm{ad}}
\def\R{\mathbb{R}}

\newcommand{\p}{\partial}
\maketitle
\noindent 
\abstract{\noindent 
In \cite{CCN} the authors developed a method for computing normal forms of dynamical systems with a coupled cell network structure. We now apply this theory to one-parameter families of homogeneous feed-forward chains with $2$-dimensional cells. Our main result is that Hopf bifurcations in such families generically generate branches of periodic solutions with amplitudes growing like $\sim |\lambda|^{\frac{1}{2}},  \sim |\lambda|^\frac{1}{6}, \sim |\lambda|^{\frac{1}{18}}$, etc. Such amplified Hopf branches were previously found in a subclass of feed-forward networks with three cells, first under a normal form assumption \cite{golstew} and later by explicit computations \cite{elmhirst}, \cite{claire2}. We explain here how these bifurcations arise generically in a broader class of feed-forward chains of arbitrary length.
}
\section{Introduction}\label{sec:2}
In this paper we shall consider systems of ordinary differential equations of the following homogeneous feed-forward type:
\begin{align}\label{network}
\begin{array}{lll}
\dot{x}_0 & = & f(x_0,x_0,x_0,\cdots, x_0, x_0; \lambda),\\
\dot{x}_1 & = & f(x_1,x_0,x_0,\cdots, x_0, x_0; \lambda), \\
\ \vdots &\, \,  \vdots & \hspace{1cm}  \vdots  \\
\dot{x}_{n-1} & = & f(x_{n-1},x_{n-2},\cdots, x_0, x_0; \lambda),\\
\dot{x}_n & = & f(x_n\ \ \ ,x_{n-1},\cdots, x_1, x_0; \lambda).
\end{array}
\end{align}
Here $n +1\in \N$ is the length of the feed-forward chain, the state variables $x_0, \ldots, x_n\in V$ are elements of a finite dimensional vector space $V$ and the function $f:V^{n+1}\times\R\to V$ is a parameter dependent response function. We shall assume that $f(0;\lambda)=0$ for all $\lambda$, and hence that equations (\ref{network}) admit a fully synchronous steady state solution $x=(0, \ldots, 0)$ for all values of the parameter. 

We are interested in the periodic solutions that emanate from this synchronous steady state as $\lambda$ varies. In order to find such synchrony breaking bifurcations of periodic solutions, let us denote by $\mathfrak{a}_i=\mathfrak{a}_i(\lambda):= D_if(0;\lambda): V\to V$. Then the linearization matrix of (\ref{network}) at the synchronous equilibrium reads
\begin{align}\label{matrix}
\left( \begin{array}{rcccc} 
\mathfrak{a}_0 + \mathfrak{a}_1+\ldots + \mathfrak{a}_n & 0 & \cdots & 0& 0\\
\mathfrak{a}_1 + \ldots + \mathfrak{a}_n & \mathfrak{a}_0 &\cdots & 0& 0\\
\vdots\   & \vdots & \ddots & \vdots &  \vdots \\
\mathfrak{a}_{n-1}+\mathfrak{a}_n & \mathfrak{a}_{n-2} & \cdots  & \mathfrak{a}_0& 0\\
\mathfrak{a}_n & \mathfrak{a}_{n-1} & \cdots & \mathfrak{a}_1 & \mathfrak{a}_0
\end{array} \right)\, .
\end{align}
This matrix displays a remarkable degeneracy: the eigenvalues of the submatrix $\mathfrak{a}_0$ each occur at least $n$ times as eigenvalues of the linearization matrix (\ref{matrix}). Although such a \(1:1:\cdots:1\) resonance is highly uncommon in differential equations without any special structure, it is generic in the context of our feed-forward network. 

Assume for example that $\dim V=2$ and that the eigenvalues of $\mathfrak{a}_0(\lambda)$ form a complex conjugate pair that crosses the imaginary axis at nonzero speed. Then one may expect a very particular $n$-fold Hopf bifurcation to take place in equations (\ref{network}). Theorem \ref{mainthm} below describes this bifurcation. It is the main result of this paper. 

\begin{theorem}\label{mainthm}
Let $V= \R^2\cong \C$ and $f:V^{n+1}\times \R \to V$ a smooth function. Assume that $f(0;\lambda)=0$ and that $\mathfrak{a}_0(0)=D_0f(0;0): V\to V$ has eigenvalues $\pm i\omega_0\neq 0$. Then under otherwise generic conditions on $f(X; \lambda)$, the local normal form of (\ref{network}) near $(x; \lambda)=(0;0)$ supports a branch of hyperbolic periodic solutions of the form 
$$x_0(t)\equiv 0, x_1(t)= B_1(\lambda) e^{i\omega(\lambda) t}, \ldots, x_n(t)= B_n(\lambda) e^{i\omega(\lambda)t} $$
of frequency \(\omega(\lambda)  = \omega_0 +\mathcal{O}(|\lambda|) \) and amplitudes
\( |B_i(\lambda)| \sim |\lambda|^{\kappa_i}\), where \(\kappa_i:=\frac{1}{2}\frac{1}{ 3^{i-1}}\).
\end{theorem}
The peculiar Hopf branch described in Theorem \ref{mainthm} has an $x_1$-component with amplitude growing at the rate $\sim |\lambda|^{\frac{1}{2}}$ of the ordinary Hopf bifurcation. The amplitude growth of its $x_2, x_3, \ldots$-components at rates $\sim |\lambda|^{\frac{1}{6}}, \sim |\lambda|^{\frac{1}{18}}, \ldots$ is much faster though. One could say that the feed-forward chain acts as an amplifier and it has been conjectured that this is why feed-forward chains occur so often as {\it motifs} in larger networks \cite{claire}.

Due to their hyperbolicity, the periodic solutions of Theorem \ref{mainthm} persist as true solutions of (\ref{network}). We also remark that the branch of periodic solutions given in Theorem \ref{mainthm} implies that there exist many more branches of periodic solutions.
This is because equations (\ref{network}) admit a symmetry: whenever $t\mapsto (x_0(t), \ldots, x_n(t))$ is a solution to (\ref{network}), then so is
$$t\mapsto (x_0(t), x_0(t), x_1(t), \ldots, x_{n-1}(t))\, .$$
As a result, the Hopf bifurcation of Theorem \ref{mainthm} generates a total of $n$ branches of periodic solutions (counted excluding the fully synchronous branch $x(t)\equiv 0$), given by
$$x_0(t) = \ldots = x_{r-1}(t) \equiv 0, x_{r}(t) = B_1(\lambda)e^{i\omega(\lambda)t}, \ldots, x_{n}(t)=B_{n-r+1}(\lambda)e^{i\omega(\lambda)t}\, . $$ 
Each of these branches emanates from the bifurcation point $(x;\lambda)=(0;0)$ and they all have a different synchrony type. We will see that only the branch described explicitly in Theorem \ref{mainthm} (the one with the largest amplitude and the least synchrony) can possibly be stable.


Dynamical systems with a coupled cell network structure have attracted much attention in recent years, most notably in the work of Field, Golubitsky and Stewart and coworkers. For a general overview of this theory, we refer to \cite{field}, \cite{curious}, \cite{golstew}, \cite{stewartnature} and references therein. It has been noticed by many people that networks may robustly exhibit nontrivial synchronized dynamics \cite{antonelli}, \cite{anto1}, \cite{anto2}, \cite{torok}, \cite{stewart1}, \cite{pivato}, \cite{parker}. Synchronous solutions may moreover undergo bifurcations with quite unusual features. Such synchrony breaking bifurcations have for example been studied in \cite{bifurcations}, \cite{anto4}, \cite{jeroen}, \cite{dias}, \cite{synbreak2} and \cite{synbreak}.

The anomalous Hopf bifurcation of Theorem \ref{mainthm} has also been described in the literature on coupled cell networks \cite{elmhirst}, \cite{claire2}, \cite{golstew}, namely in the context of equations of the form
\begin{align}
\begin{array}{lll}\label{network3}     
\dot x_0 & = &  g(x_0, x_0; \lambda)\, ,\\
\dot x_1 & = &  g(x_1, x_0; \lambda)\, ,\\
\dot x_2 & = &  g(x_2, x_1; \lambda)\, .
\end{array}
\end{align}
Note that (\ref{network3}) arises as a special case of (\ref{network}), with $n=2$ and $f:V^3\to V$ of the form 
$$f(X_0, X_1, X_2) :=g(X_0, X_1)\, .$$ 
In fact, the equivalent of Theorem \ref{mainthm} for equations of the form (\ref{network3}) was first proved in \cite{golstew} under a normal form assumption.  More precisely, it is assumed in \cite{golstew} that the right hand side of (\ref{network3}) commutes with the symmetry
$(x_0, x_1, x_2) \mapsto (x_0, e^{is}x_1, e^{is}x_2)$
and hence that $g(X; \lambda)$ has the symmetry
\begin{align}\label{invarianceg}
g(e^{is}X_0,e^{is}X_1; \lambda)=e^{is}g(X_0, X_1; \lambda)\ \mbox{for}\ s\in \R \, .
\end{align}
With this invariance, the periodic solutions of (\ref{network3}) become relative equilibria, which makes them computable. We remark that a generic $g(X;\lambda)$ of course does not satisfy (\ref{invarianceg}), but the existence of the anomalous Hopf branch was later confirmed for a generic $g(X;\lambda)$ by an explicit computation of the normal form of (\ref{network3}) in \cite{elmhirst}. Finally, with the use of center manifold reduction, an alternative and more efficient proof of the same result was given in \cite{claire2}. For similar normal form computations of other networks, we refer to \cite{krupa}.

Theorem \ref{mainthm} is thus an extension of the results in \cite{elmhirst}, \cite{claire2} and  \cite{golstew}, valid for a more general class of feed-forward chains of arbitrary length. More importantly, our proof of Theorem \ref{mainthm} is in the spirit of a generic bifurcation theory for coupled cell networks. It makes use of the theory developed by the authors in \cite{CCN} that explains how to compute the local normal form of a coupled cell network. 

In fact, we show in \cite{CCN} that any network architecture admits a natural Lie algebra that captures the structure and invariants of the dynamical systems with this architecture. 
This structure is that of a so-called ``semigroup network'' and the main result in \cite{CCN} is that the normal form of a semigroup network is a semigroup network as well. It turns out that
equations (\ref{network}) form an example of a semigroup network, and hence the normal form of (\ref{network}) near a steady state is again of the form (\ref{network}). This proves it justified to assume that equations (\ref{network}) are in normal form, and hence that $f(X;\lambda)$ satisfies
\begin{align}\label{invariancef}
f(e^{is}X_0,\ldots, e^{is}X_{n-1}, 0; \lambda)=e^{is}f(X_0, \ldots, X_{n-1}, 0; \lambda)\, .
\end{align}
Using the latter invariance, the analysis of the periodic solutions of (\ref{network}) is similar to the analysis in \cite{golstew}. This analysis eventually leads to Theorem \ref{mainthm}. It will also become clear that it is not reasonable to assume that the normal form of (\ref{network3}) is of the form (\ref{network3}), because equations (\ref{network3}) do not form a semigroup network. 

This paper is organized as follows. In Section \ref{sec:0}, we illustrate the amplifying character of our feed-forward chain by studying a synchrony breaking steady state bifurcation in case $\dim V =1$. This helps to fix ideas, and it yields an extension of some results obtained in \cite{CCN}. In Section \ref{sec:01}, we recall some results from \cite{CCN} on semigroup coupled cell networks. 
Sections \ref{sec:1} and \ref{semigroupringsection} are devoted to the algebra of linear feed-forward networks. The goal of these sections is to obtain, by means of linear normal form theory,
 a decomposition in semisimple and nilpotent part of the linearization matrix (\ref{matrix}). This is an essential preparatory step for the computation of a nonlinear normal form. We provide a linear normal form theorem in Section \ref{sec:1} and prove it in Section \ref{semigroupringsection}, using the concept of a semigroup ring. 
Finally, in Section \ref{sec:4} we use the \(SN\)-decomposition obtained in Sections \ref{sec:1} and \ref{semigroupringsection} to compute the nonlinear normal form of (\ref{network}) under the assumptions of Theorem \ref{mainthm}. A singularity analysis then leads to the proof of this theorem.


\section{An amplified steady state bifurcation}\label{sec:0}
In this section we describe a synchrony breaking steady state bifurcation in the feed-forward network (\ref{network}) that confirms its amplifying character. This section is meant as an illustration before we start the actual proof of Theorem \ref{mainthm}. Most notably, we do not make use of normal forms in this section yet.

We first of all remark that when $f(0; 0) = 0$, then equations (\ref{network}) admit a fully synchronous steady state $x=(0, \ldots, 0)$ at the parameter value $\lambda=0$. This steady state persists to a nearby synchronous steady state $(s(\lambda), \ldots, s(\lambda))$ for $\lambda$ close to $0$ under the condition that 
$$\left. \frac{d}{ds}\right|_{s=0}\!\!\! f(s, \ldots, s; 0)=\mathfrak{a}_0(0) + \ldots + \mathfrak{a}_n(0) \ \mbox{is invertible}.$$ 
Throughout this paper, we will assume that this condition is met, so that by translating to the synchronous equilibrium we justify our assumption in the introduction that $f(0; \lambda)=0$. 

We are interested in branches of solutions that emanate from $(x; \lambda)=(0;0)$. These are defined as follows:
\begin{definition}
A {\it branch of steady states} near $(x;\lambda)=(0;0)$ is the germ of a continuous map $x: [0, \lambda_0)\to V^{n+1}$ or $x:(-\lambda_0, 0]\to V^{n+1}$ with the properties that $x(0)=0$ and such that each $x(\lambda)$ is a steady state of the differential equations at the parameter value $\lambda$. 
\end{definition}
To characterize the growth of a branch of steady states, we introduce a special symbol:
\begin{definition}
For a branch $x(\lambda)$ of steady states we write $x(\lambda)\sim |\lambda|^{\kappa}$, with $\kappa>0$, if there is a smooth function $x^*:[0, |\lambda_0|^{\kappa}) \to V^{n+1}$ such that $$x(\lambda) = x^*(|\lambda|^{\kappa})\ \mbox{and}\ \lim_{|\lambda|^{\kappa}\downarrow 0} \frac{x^*(|\lambda|^{\kappa})}{|\lambda|^{\kappa}} \neq 0\, .$$
\end{definition}
The following theorem describes the branches of steady states that can bifurcate in (\ref{network}) in the case that $\dim V=1$. It is a generalization of results in \cite{CCN}, where it was assumed that $n=2$. We spell out the proof of Theorem \ref{mainthm1} in great detail and hope that this makes the structure of the proof of Theorem \ref{mainthm} more understandable.
\begin{theorem}\label{mainthm1}
Let $V= \R$ and $f:V^{n+1}\times \R \to V$ a smooth function. Assume that $f(0;\lambda)=0$ and that $\mathfrak{a}_0(0)=0$. Furthermore, assume the following generic conditions on $f(X; \lambda)$:
\begin{itemize}
\item[1.] $\mathfrak{a}_0(0)+\ldots +\mathfrak{a}_n(0)\neq 0$.
\item[2.] $\frac{d\mathfrak{a}_0}{d\lambda}(0)\neq 0$.
\item[3.] $\mathfrak{a}_1(0)\neq 0$.
\item[4.] $\frac{\partial^2 f}{dX_0^2}(0;0)\neq 0$.
\end{itemize}
Then equations (\ref{network}) support $2n$ branches of hyperbolic steady states (counted excluding the trivial steady state $x = 0$) near $(x; \lambda)=(0;0)$. More precisely, if we define $\kappa_i:=\frac{1}{ 2^{i-1}}$ for $i=1,\ldots,n$, then for each $r=1, \ldots, n$ there are two steady state branches with asymptotics
$$x_0= \ldots = x_{r-1}= 0, x_{r} \sim |\lambda|^{\kappa_1}, \ldots, x_n\sim |\lambda|^{\kappa_{n-r+1}}\, .$$ 
\end{theorem}
\begin{proof}
Let us write $\mathfrak{a}_0'(0)=\frac{d\mathfrak{a}_0}{d\lambda}(0)$. We will give the proof of the theorem under the assumption that $\mathfrak{a}_0'(0)\mathfrak{a}_1(0)>0$. The proof is similar when $\mathfrak{a}_0'(0)\mathfrak{a}_1(0)<0$ but it is inconvenient to give both proofs simultaneously. This technical problem will not occur in the proof of Theorem \ref{mainthm}.

Because $\mathfrak{a}_0(0)+\ldots + \mathfrak{a}_n(0)\neq 0$ we can assume that $f(0;\lambda)=0$. Together with the remaining assumptions of the theorem, this implies that we can Taylor expand 
\begin{align}
 f(X;\lambda) = & \mathfrak{a}_0'(0) \lambda X_0 + \mathfrak{a}_1(0)X_1+ \ldots+ \mathfrak{a}_n(0)X_n + C X_0^2  \nonumber \\ \label{expansion} & +  \mathcal{O} \left(|X_0|^3 + |X_0|^2|\lambda| + |X_0||\lambda|^2 + ||(X_1, \ldots, X_n)||^2 \right. \\ \nonumber & + \left.||(X_1, \ldots, X_n)||\cdot |\lambda| + |X_0|\cdot||(X_1, \ldots, X_n)||\right)\, ,
\end{align}
in which $\mathfrak{a}_0'(0)\neq0$, $\mathfrak{a}_1(0)\neq 0$, $\mathfrak{a}_1(0)+\ldots + \mathfrak{a}_n(0)\neq 0$ and $C:=\frac{\partial^2f}{\partial X_0^2}(0;0) \neq 0$. 
With this in mind, we will solve for steady states by consecutively solving for $x_j$ the equations 
\begin{align}\nonumber
 f(x_j, x_{j-1}, \ldots, x_1, x_0, \ldots, x_0; \lambda)=  0 \ \mbox{for}\ j=0,\ldots, n \, .
\end{align} 
First of all, since $\mathfrak{a}_1(0)+\ldots + \mathfrak{a}_n(0)\neq 0$, it holds by (\ref{expansion}) that 
$$f(x_0, \ldots, x_0; \lambda)=(\mathfrak{a}_1(0)+\ldots+\mathfrak{a}_n(0))x_0+\mathcal{O}(|x_0|^2+|\lambda| |x_0|)\, .$$ 
We conclude that $x_0=x_0^*:=0$ must hold for a steady state. 
 
In fact, it is clear for any $1\leq r \leq n$ that $x_0=x_0^*=0, \ldots, x_{r-1}=x_{r-1}^*:=0$ provide solutions to the equations $f(x_0, \ldots, x_0; \lambda)=0, \ldots, f(x_{r-1}, \ldots, x_0; \lambda)=0$.  Given these trivial solutions, let us try to find the possible steady state values of $x_r$ by solving 
\begin{align}\label{x1eq}
f(x_r, x_{r-1}^*, \ldots, x_0^*; \lambda) = \mathfrak{a}_0'(0)\lambda x_r + C x_r^2+ \mathcal{O}(|x_r|^3 + |x_r|^2|\lambda| + |x_r|\cdot |\lambda|^2)=0\, .
\end{align} 
Not surprisingly, one solution to this equation is given by $x_r=0$, but let us ignore this solution. Thus, we divide equation (\ref{x1eq}) by $x_r$ and consider the case that $x_r$ solves 
$$\mathfrak{a}_0'(0)\lambda+C x_r + \mathcal{O}(|x_r|^2 + |x_r|\cdot |\lambda| + |\lambda|^2)=0\, .$$ 
Because $C \neq 0$, the implicit function theorem guarantees a unique solution $x_r$ to this latter equation, given by
$$x_r = -\frac{\mathfrak{a}_0'(0)}{C}\lambda +\mathcal{O}(\lambda^2)\, .$$
Note that $x_r$ is defined for both positive and negative values of $\lambda$. Thus, we find two functions $x_r^{\pm*}:[0, \lambda_0)\to V$ such that $x_r(\lambda) = x_r^{\pm*}(|\lambda|)\sim |\lambda|^{\frac{1}{2}}= |\lambda|^{\kappa_1}$ solve (\ref{x1eq}). In particular, taking $r=n$, we proved the existence of two of the branches in the statement of the theorem, namely those with asymptotics $x_0 = \ldots = x_{n-1}=0$ and $x_n\sim |\lambda|^{\kappa_1}$. 

In case $1\leq r\leq n-1$, we shall ignore the branch $x_r^{-*}(|\lambda|)$ defined for negative $\lambda$: it leads to an equation for $x_{r+1}$ that can not be solved. This follows from computations similar to those given below for the positive branch. We will not provide these considerations in detail. We shall write $x_r^*=x_r^{+*}$.

Concerning this positive branch, let us remark that $\lim_{\lambda\downarrow 0} \frac{x_r^*(\lambda)}{\lambda} = -\frac{\mathfrak{a}_0'(0)}{C}$ and $\frac{\mathfrak{a}_1(0)}{C}$ have opposite sign due to our assumption that $\mathfrak{a}_0'(0)\mathfrak{a}_1(0)>0$. This leads, for $r\leq j\leq n-1$, to the following two induction hypotheses:
\begin{itemize}
\item[1.]  Assume that we found for all $i=r,\ldots, j$ certain 
smooth functions $x_i^*:[0, \lambda_0^{\kappa_{i-r+1}})\to\R$ so that $$x_0\!=\!x_0^* \!=\!0, \dots, x_{r-1}\!=\!x_{r-1}^*\!=\!0, x_r\!=\! x^*_r(\lambda^{\kappa_1})\!\sim\! \lambda^{\kappa_1}, \ldots , x_j\!=\! x_j^*(\lambda^{\kappa_{j-r+1}})\! \sim\! \lambda^{\kappa_{j-r+1}}$$ 
solve the equations $f(x_0, \ldots, x_0;\lambda)=0, \ldots, f(x_j, x_{j-1}, \ldots, x_0, \ldots, x_0;\lambda)=0$.
\item[2.] The sign of $\lim_{\lambda\downarrow 0} \frac{x_j^*(\lambda^{\kappa_{j-r+1}})}{\lambda^{\kappa_{j-r+1}}}$ is opposite to the sign of $\frac{\mathfrak{a}_1(0)}{C}$.
\end{itemize}
We remark that we just proved that these induction hypotheses are true for $j=r$. We will now try obtain $x_{j+1}$ by solving the equation $f(x_{j+1}, x_{j}^*, \ldots, x_0^*, \ldots, x_0^*;\lambda)=0$. 

Anticipating that the steady state value of $x_{j+1}$ will be smoothly depending on $\lambda^{\kappa_{j-r+2}}$, let us at this point define the rescaled parameter $\mu:=\lambda^{\kappa_{j-r+2}} = \lambda^{\frac{1}{2^{j-r+1}}}$ and the rescaled unknown $x_{j+1}=: \mu y_{j+1}$. Then it holds that $\mu^{2^{j-i+1}} = \lambda^{\kappa_{i-r+1}}$, which inspires us to define also the rescaled functions $y_r^*, \ldots, y_j^*$ by
$$x_r^*(\lambda^{\kappa_1}) =: \mu^{2^{j-r+1}}y_r^*(\mu^{2^{j-r+1}}), \ldots, x_j^*(\lambda^{\kappa_{j-r+1}})=:\mu^2y_j^*(\mu^2)\ \mbox{for} \ \mu=\lambda^{\kappa_{j-r+2}} \, .$$ 
By the first induction hypothesis, the functions $y_i^*$ are smooth and $y_j^*(0)=\lim_{\lambda\downarrow 0}\frac{x_j^*(\lambda^{\kappa_{j-r+1}})}{\lambda^{\kappa_{j-r+1}}}$ $\neq 0$. Moreover, using (\ref{expansion}), one checks that in terms of the rescaled variables, the equation for $x_{j+1}=\mu y_{j+1}$ takes the form
$$\mu^2 \left( \mathfrak{a}_1(0)y_{j}^*(0) + Cy_{j+1}^2\right) + \mathcal{O}(|y_{j+1}|\cdot |\mu|^3 + |\mu|^{4}) = 0\, .$$
Dividing this by $\mu^2$, we find that we need to solve an equation of the form
$$h(y_{j+1}; \mu) = \left( \mathfrak{a}_1(0)y_{j}^*(0) + Cy_{j+1}^2\right) + \mathcal{O}(|y_{j+1}|\cdot |\mu| + |\mu|^{2}) = 0\, .$$
Recall that $\mathfrak{a}_1(0), y_{j}^*(0), C(0)\neq 0$. In fact, by the second induction hypothesis it holds that $-\frac{\mathfrak{a}_1(0)y_j^*(0)}{C}>0$. Let $Y_{j+1}:=\sqrt{-\frac{\mathfrak{a}_1(0)y_j^*(0)}{C}}$ so that $\pm Y_{j+1}$ are the solutions to the equation $h(Y_{j+1}, 0)=\mathfrak{a}_1(0)y_j^*(0) + CY_{j+1}^2 = 0$. Then clearly $D_{y_{j+1}}h(\pm Y_{j+1},0)=\pm 2CY_{j+1}\neq 0$ and thus by the implicit function theorem there exist smooth functions $y_{j+1}^{\pm*}(\mu) = \pm Y_{j+1}+\mathcal{O}(\mu)$ solving $h(y^{\pm*}_{j+1}(\mu), \mu)=0$. Correspondingly, the expressions
$$x_{j+1}=(x_{j+1}^{*})^{\pm}(\lambda^{\kappa_{j-r+2}}) = \mu y_{j+1}^{\pm*}(\mu) = \pm \lambda^{\kappa_{j-r+2}}Y_{j+1} + \mathcal{O}(\lambda^{\kappa_{j-r+1}}) \sim \lambda^{\kappa_{j-r+2}}$$
provide two branches of solutions to the equation $f(x_{j+1}, x_j^* \ldots, x_0^*; \lambda)=0$. It holds that $\lim_{\lambda\downarrow 0} \frac{x_{j+1}^{\pm*}(\lambda^{\kappa_{j-r+2}})}{\lambda^{\kappa_{j-r+2}}}=y_{j+1}^{*\pm}(0)= \pm Y_{j+1}$ so that for precisely one of these branches the sign of this limit is opposite to the sign of $\frac{\mathfrak{a}_1(0)}{C}$. Only this branch can be used in the next step of the induction. This step is necessary precisely when $j+1\leq n-1$. This finishes the induction and the proof that also for $1\leq r\leq n-1$ there exist two
 steady state branches with asymptotics
$$x_0=0, \ldots, x_{r-1}=0, x_r \sim \lambda^{\kappa_1}, \ldots, x_n \sim \lambda^{\kappa_n}\ .$$
Note that these two branches only differ in their values of $x_n$.

We remark that in case $\mathfrak{a}_0'(0)\mathfrak{a}_1(0)<0$, the branches exist for negative values of $\lambda$.
Finally, we consider the linearization matrix around a steady state on the $r$-th branch. This matrix is of the lower triangular form
\begin{align}\nonumber
\left( \begin{array}{ccccccc} 
\mathfrak{a}_0(\lambda) + \mathfrak{a}_1(\lambda)+\ldots + \mathfrak{a}_n(\lambda) & 0 & & \hdots&  & & 0\\
* & \mathfrak{a}_0(\lambda) &0 & &  & \hdots & 0\\
   & * & \ddots & &  & & \vdots \\
&  &  & \mathfrak{a}_0(\lambda) & 0 & & 0\\
\vdots & \vdots & & * & \mathfrak{b}_1(\lambda) & 0& 0 \\
& & & & & \ddots & \vdots \\
* & * & & \cdots & & * & \mathfrak{b}^{\pm}_{n-r+1}(\lambda)
\end{array} \right)\, .
\end{align}
The eigenvalues of this linearization matrix are $\mathfrak{a}_0(\lambda)+ \ldots+\mathfrak{a}_n(\lambda)$, $\mathfrak{a}_0(\lambda)$ ($r-1$ times), $\mathfrak{b}_1(\lambda), \ldots, \mathfrak{b}_{n-r}(\lambda)$ and $\mathfrak{b}_{n-r+1}^{\pm }(\lambda)$, where 
$$\mathfrak{b}_j(\lambda)=D_{0}f(x_{r+j-1}^*,\ldots , x_0^*) = 2Cx_{r+j-1}^*+ \mathcal{O}(|\lambda|^{\kappa_{j-1}})\sim |\lambda|^{\kappa_j}\, .$$
for $j=1,\ldots, n-r$ and 
$$\mathfrak{b}_{n-r+1}^{\pm}(\lambda)= D_0f(x_{n}^{\pm*}, \ldots , x_0^*) = 2Cx_{n}^{\pm*}+ \mathcal{O}(|\lambda|^{\kappa_{n-r}})\sim |\lambda|^{\kappa_{n-r+1}}\, .$$
For $\lambda\neq 0$, these eigenvalues are real and nonzero, i.e. the branches are hyperbolic. 
\end{proof}
\begin{remark}
Only one of the branches of solutions given in Theorem \ref{mainthm1} can possibly be stable and this is one of the two branches with the least synchrony and the largest amplitude, i.e. one of the two branches with asymptotics
$$x_0=0, x_1\sim |\lambda|^{\kappa_1}, \ldots, x_n\sim |\lambda|^{\kappa_n}\, .$$
This happens precisely when $\mathfrak{a}_0(0)+\ldots+\mathfrak{a}_n(0)<0$ and $\mathfrak{a}_1(0)>0$. We leave the proof of this claim to the interested reader.
 \end{remark}

\section{A semigroup network}\label{sec:01}
The feed-forward differential equations (\ref{gammaf}) form an example of a so-called {\it semigroup network}. These networks were defined by the authors in \cite{CCN}, and they have the remarkable property that the Lie bracket of two semigroup networks is again a semigroup network. 

In the context of our feed-forward chain, this is perhaps best explained as follows. First of all, note that we can write the differential equations (\ref{network}) as 
$$\dot x = \gamma_f(x)\ \mbox{for}\ x\in V^{n+1}\, ,$$
where the vector field $\gamma_f:V^{n+1}\to V^{n+1}$ is defined for a function $f:V^{n+1}\to V$ as
\begin{align}\label{gammaf}
(\gamma_f)_j(x):=f(x_{\sigma_0(j)}, x_{\sigma_1(j)}, \ldots, x_{\sigma_{n-1}(j)}, x_{\sigma_n(j)})\ \mbox{for}\ j=0, \ldots, n\, .
\end{align}
Here, $\sigma_0, \ldots, \sigma_n$ are maps from $\{0, \ldots, n\}$ to $\{0, \ldots, n\}$, given by 
$$\sigma_{i}(j):={\rm max}\, \{j-i, 0\}\, .$$
One can now observe that for all $0\leq i_1,i_2\leq n$ it holds that
$$\sigma_{i_1} \sigma_{i_2}:=\sigma_{i_1}\circ \sigma_{i_2}=\sigma_{{\rm min}\{i_1+i_2, n\}}\, .$$
This means in particular that the collection 
$$\Sigma:=\{\sigma_0, \sigma_1, \ldots, \sigma_n\}$$
 is closed under composition: it is a semigroup. In fact, $\Sigma$ is commutative and generated by the elements $\sigma_0$ and $\sigma_1$. Moreover, the elements $\sigma_0$ and $\sigma_n$ are somewhat special: $\sigma_0$ is the unit of $\Sigma$ and $\sigma_n$ plays the role of ``zero'', because $\sigma_n\sigma_i=\sigma_i\sigma_n=\sigma_n$ for all $0\leq i\leq n$.

All of this leads to the following result, that was proved for general semigroup networks in \cite{CCN}:
\begin{theorem}\label{brackettheorem}
Define for $0\leq i \leq n$ the linear map $A_{\sigma_i}: V^{n+1}\to V^{n+1}$ by
\begin{align}\nonumber
(A_{\sigma_i}X)_j&:= X_{{\rm min} \, \{j+i, n\}}\, .
\end{align}
Then it holds that 
$$A_{\sigma_{i_1}}\circ A_{\sigma_{i_2}}=A_{\sigma_{i_1}\sigma_{i_2}}\ \mbox{and that}\ A_{\sigma_i}(x_{\sigma_0(j)}, \ldots, x_{\sigma_n(j)}) = (x_{\sigma_0(\sigma_i(j))}, \ldots, x_{\sigma_n(\sigma_i(j))})\, .$$
Moreover, for any $f, g\in C^{\infty}(V^{n+1},V)$ it holds that $$[\gamma_f, \gamma_g] = \gamma_{[f,g]_{\Sigma}}$$
where
$$[f,g]_{\Sigma}(X) := \sum_{i=0}^n D_if(X)\cdot g(A_{\sigma_i}X) - D_ig(X)\cdot f(A_{\sigma_i}X)\, .$$ 
\end{theorem}
\begin{proof}
It is clear that $$(A_{\sigma_{i_1}}A_{\sigma_{i_2}}X)_j = X_{{\rm min}\, \{j+i_1+i_2, n\}} = (A_{\sigma_{{\rm min}\, \{i_1+i_2, n\}}}X)_j=(A_{\sigma_{i_1} \sigma_{i_2}}X)_j\, .$$
The fact that $A_{\sigma_i}(x_{\sigma_0(j)}, \ldots, x_{\sigma_n(j)}) = (x_{\sigma_0(\sigma_i(j))}, \ldots, x_{\sigma_n(\sigma_i(j))})$ is obvious from our definitions. This proves the first claim of the theorem.

Next, recall that $[\gamma_f, \gamma_g](x)= D\gamma_f(x)\cdot \gamma_g(x) - D\gamma_g(x)\cdot \gamma_f(x)$. One computes that
\begin{align}\nonumber
(D\gamma_g(x)\cdot \gamma_f(x))_j &= \sum_{k=0}^n D_k(\gamma_f(x))_j\cdot (\gamma_g)_k(x) = \\ \nonumber 
& \sum_{k=0}^n  \frac{d}{dx_k}f(x_{\sigma_0(j)}, \ldots, x_{\sigma_n(j)}) \cdot g(x_{\sigma_0(k)}, \ldots, x_{\sigma_n(k)}) = \\ \nonumber
&  \sum_{i=0}^n D_if(x_{\sigma_0(j)}, \ldots, x_{\sigma_n(j)}) \cdot g(x_{\sigma_0(\sigma_i(j))}, \ldots, x_{\sigma_n(\sigma_i(j))}) = \\ \nonumber
&  \sum_{i=0}^n D_if(x_{\sigma_0(j)}, \ldots, x_{\sigma_n(j)}) \cdot g(A_{\sigma_i}(x_{\sigma_0(j)}, \ldots, x_{\sigma_n(j)})) \, .
\end{align}
With a similar computation for $(D\gamma_g(x)\cdot \gamma_f(x))_j$, we thus find that
$$ [\gamma_f, \gamma_g]_j(x) = \left. \sum_{i=0}^n D_if(X) \cdot g(A_{\sigma_i}X) -  D_ig(X) \cdot f(A_{\sigma_i}X) \right|_{X=(x_{\sigma_0(j)}, \ldots, x_{\sigma_n(j)})}\, . $$ 
This proves the theorem.
\end{proof}
The first statement of Theorem \ref{brackettheorem} is that the map 
$$\sigma_i\mapsto A_{\sigma_i}\, , \ \Sigma \to \mathfrak{gl}(V^{n+1})$$
is a representation of $\Sigma$. The second and more important statement is that the Lie bracket of the two feed-forward vector fields $\gamma_f$ and $\gamma_g$ is another feed-forward vector field of the same form, namely $\gamma_{[f, g]_{\Sigma}}$. Moreover, the new response function $[f,g]_{\Sigma}$ is computed from $f$ and $g$ with the use of the representation $\sigma_i\mapsto A_{\sigma_i}$. 

The most important consequence of Theorem \ref{brackettheorem} is that the collection 
$$\gamma(C^{\infty}(V^{n+1},V))=\{\gamma_f\, |\, f\in C^{\infty}(V^{n+1}, V)\}$$
of feed-forward vector fields is a Lie algebra. This implies for example that there exists a large class of transformations of the phase space $V^{n+1}$ that leaves the class of feed-forward vector fields invariant: the time-$1$ flow $e^{\gamma_g}$ of any feed-forward vector field $\gamma_g$, will transform the feed-forward  vector field $\gamma_f$ into another feed-forward vector field, namely:
$$(e^{\gamma_g})_*\gamma_f= e^{\ad_{\gamma_g}}(\gamma_f) = \gamma_f+[\gamma_g, \gamma_f] + \frac{1}{2}[\gamma_g, [\gamma_g, \gamma_f]] + \ldots = \gamma_{f + [g, f]_{\Sigma} + \frac{1}{2}[g, [g,f]_{\Sigma}]_{\Sigma}+\ldots}\, .$$
This explains why transformations of $V^{n+1}$ of the form $e^{\gamma_g}$ play an important role in the theory of local normal forms of semigroup networks. Theorem \ref{normalformtheoremparameters} below, for example, was proved in \cite{CCN}. It describes the normal forms of one-parameter families of feed-forward networks. To formulate it, we define for $k, l\geq 0$,
$$P^{k,l}:=\{f:V^{n+1}\times\R\to V \ \mbox{homogeneous polynomial of degree} \ k+1\ \mbox{in}\ X\ \mbox{and degree}\ l \ \mbox{in}\ \lambda\}\, .$$ 
\begin{theorem}[Normal form theorem]
 \label{normalformtheoremparameters}
Let $f\in C^{\infty}(V^{n+1}\times\R, V)$ and assume that $f(0;\lambda)=0$. Let us write the Taylor expansion of $f$ as 
$$f= (f_{0,0} + f_{0,1}+ f_{0,2} + \ldots) + (f_{1,0} + f_{1,1} + f_{1,2}+\ldots) + \ldots$$ 
with $f_{k,l}\in P^{k,l}$. We moreover denote by 
$$A:=D_x\gamma_f(0;0) = \gamma_{f_{0,0}}=A_S + A_N: V^{n+1}\to V^{n+1}$$
the $SN$-decomposition of the linearization $D_x\gamma_f(0;0)$. Finally, let $1\leq r_1, r_2<\infty$.

Then there exists a polynomial family $\lambda\mapsto \Phi(\cdot; \lambda)$ of analytic diffeomorphisms, defined for $\lambda$ in an open neighborhood of $0$ and each sending an open neighborhood of $0$ in $V^{n+1}$ to an open neighborhood of $0$ in $V^{n+1}$, such that $\Phi(\cdot; \lambda)$ conjugates $\gamma_{f(\cdot; \lambda)}$ to $\gamma_{\overline f(\cdot;\lambda)}$, where
$$ \overline f = (f_{0,0} + \overline f_{0,1}+ \overline f_{0,2} + \ldots) + (\overline f_{1,0} + \overline f_{1,1} + \overline f_{1,2}+\ldots) + \ldots $$
has the property that for all $ 0\leq k\leq r_1$ and $0\leq l\leq r_2$,
\begin{align}\label{nfsymmetry}
e^{s A_S} \circ \gamma_{\overline f_{k,l}} = \gamma_{\overline f_{k,l}}  \circ e^{sA_S}\ \mbox{for all}\ s\in \R\, .
\end{align}
\end{theorem}
\begin{proof}{\bf \ [Sketch]}
This proof is based on the fact that the spaces $P^{k,l}$ are graded, that is 
$$[P^{k,l}, P^{K,L}]_{\Sigma} \subset P^{k+K, l+L}\, .$$ 
As a consequence, one can start by choosing an $g_{0,1}\in P^{0,1}$ and use the time-$1$ flow $e^{\gamma_{g_{0,1}}}$ of $\gamma_{g_{0,1}}$ to transform $\gamma_{f}=\gamma_{f_{0,0}+f_{0,1}+\ldots}$ into
$$(e^{\gamma_{g_{0,1}}})_*\gamma_f = \gamma_{f+[g_{0,1},f]_{\Sigma}+\ldots}=\gamma_{f_{0,0}+ (f_{0,1}+[g_{0,1}, f_{0,0}]_{\Sigma}) +\ldots}\, .$$
By choosing $g_{0,1}$ appropriately, one can then make sure that $\gamma_{\overline f_{0,1}}:=\gamma_{f_{0,1}+[g_{0,1}, f_{0,0}]_{\Sigma}}$ commutes with the semisimple part $A_S$ of $\gamma_{f_{0,0}}$. This is a consequence of the fact that both $A_S$ and $A_N$ are semigroup networks. The precise argument leading to this result is nontrivial though and is given in detail in \cite[Section 9]{CCN}.

Next, one normalizes $f_{0,2}, \ldots, f_{0,r_2}, f_{1,0}, \ldots, f_{1, r_2}, \ldots, f_{r_1,0}, \ldots, f_{r_1, r_2}$. The grading of the $P^{k,l}$ ensures that, once $f_{k,l}$ has been normalized into ${\overline f}_{k,l}$, it is not changed anymore by any subsequent transformations.
\end{proof}
In short, Theorem \ref{normalformtheoremparameters} says that we can arrange that the local normal form of a parameter family of feed-forward vector fields $\gamma_f$ is another parameter family of feed-forward vector fields $\gamma_{\overline f}$. Moreover, this normal form may be assumed to admit a ``normal form symmetry'': it commutes with the flow of the semisimple part $A_S$ of the linearization $D_x\gamma_f(0;0) = \gamma_{f_{0,0}}$. 


\section{A linear normal form}\label{sec:1}
Theorem \ref{normalformtheoremparameters} describes the normalization of the parameter dependent coupled cell network vector field $\gamma_f$ with respect to the semisimple part of the linearization $$A:=D_x\gamma_f(0;0)=\gamma_{f_{0,0}}\, .$$ 
We recall that $f_{0,0}(X)=\mathfrak{a}_0(0)X_0 + \cdot + \mathfrak{a}_n(0)X_n$ for certain 
$$\mathfrak{a}_0(0), \ldots, \mathfrak{a}_n(0)\in \lieg:=\mathfrak{gl}(V)$$
and hence that the matrix of $A$ is given by
\begin{align}\label{matrixagain}
A=\left( \begin{array}{rcccc} 
\mathfrak{a}_0(0) + \mathfrak{a}_1(0)+\ldots + \mathfrak{a}_n(0) & 0 & \cdots & 0& 0\\
\mathfrak{a}_1(0) + \ldots + \mathfrak{a}_n(0) & \mathfrak{a}_0(0) &\cdots & 0& 0\\
\vdots\   & \vdots & \ddots & \vdots &  \vdots \\
\mathfrak{a}_{n-1}(0)+\mathfrak{a}_n(0) & \mathfrak{a}_{n-2}(0) & \cdots  & \mathfrak{a}_0(0)& 0\\
\mathfrak{a}_n(0) & \mathfrak{a}_{n-1}(0) & \cdots & \mathfrak{a}_1(0) & \mathfrak{a}_0(0)
\end{array} \right)\, .
\end{align}
To determine the semisimple part of this matrix, we will bring $A$ in ``linear coupled cell network normal form''. This linear normal form is described in the following theorem:
\begin{theorem}\label{linearnormalform}
Let $A$ be the $(n+1)\times (n+1)$ matrix given in (\ref{matrixagain}) and assume that \(\mathfrak{a}_0(0)\) is semisimple.
Then there exist linear maps $g_1, \ldots, g_{n-1}:V^{n+1}\to V$ of the form
$$g_i(X)=\mathfrak{b}_i(X_i-X_n) \ \mbox{with}\ \mathfrak{b}_i\in \lieg$$ 
so that the consecutive time-$1$ flows of the linear maps $\gamma_{g_i}:V^{n+1}\to V^{n+1}$
transform $A$ into 
\begin{align}
\overline A := \left(e^{\gamma_{g_{n-1}}}\circ \ldots \circ  e^{\gamma_{g_{1}}} \right) \circ A\circ \left(e^{-\gamma_{g_{1}}}\circ \ldots \circ  e^{-\gamma_{g_{n-1}}} \right)=\ \ \  &
\nonumber \\ \nonumber
\left( \begin{array}{rcccc} 
\mathfrak{a}_0(0) + \bar{\mathfrak{a}}_1(0)+\ldots + \bar{\mathfrak{a}}_n(0) & 0 & \cdots & 0& 0\\
\bar{\mathfrak{a}}_1(0) + \ldots + \bar{\mathfrak{a}}_n(0) & \mathfrak{a}_0(0) &\cdots & 0& 0\\
\vdots\   & \vdots & \ddots & \vdots &  \vdots \\
\bar{\mathfrak{a}}_{n-1}(0)+\bar{\mathfrak{a}}_n(0) & \bar{\mathfrak{a}}_{n-2}(0) & \cdots  & \mathfrak{a}_0(0)& 0\\
\bar{\mathfrak{a}}_n(0) & \bar{\mathfrak{a}}_{n-1}(0) & \cdots & \bar{\mathfrak{a}}_1(0) & \mathfrak{a}_0(0)
\end{array} \right)&
\end{align}
for which it holds that 
$$[\mathfrak{a}_0(0), \bar{\mathfrak{a}}_i(0)]=\mathfrak{a}_0(0) \bar{\mathfrak{a}}_i(0)-\bar{\mathfrak{a}}_i(0) \mathfrak{a}_0(0)=0\ \mbox{for all}\ 1\leq i\leq n-1\, .$$
As a consequence, $\overline A$ admits a decomposition $\overline A=\overline A_S+\overline A_N$ with 
$$\overline A_S  =  \left( \begin{array}{rcccc} 
\mathfrak{a}_0(0) + \bar{\mathfrak{a}}_1(0)+\ldots + \bar{\mathfrak{a}}_n(0) & 0 & \cdots & 0& 0\\
\bar{\mathfrak{a}}_1(0) + \ldots + \bar{\mathfrak{a}}_n(0) & \mathfrak{a}_0(0) &\cdots & 0& 0\\
\vdots\   & \vdots & \ddots & \vdots &  \vdots \\
\bar{\mathfrak{a}}_{1}(0) + \ldots  + \bar{\mathfrak{a}}_n(0) & 0 & \cdots  & \mathfrak{a}_0(0)& 0\\
\bar{\mathfrak{a}}_1(0)+ \ldots + \bar{\mathfrak{a}}_n(0) & 0 & \cdots & 0 & \mathfrak{a}_0(0)
\end{array} \right)
$$
such that the map $\overline A_N$ is nilpotent and 
$$[\overline A_S, \overline A_N] = \overline A_S \overline A_N - \overline A_N \overline A_S =0\, .$$ 
The map $\overline A_S$ is semisimple if and only if $\mathfrak{a}_0(0)+\bar{\mathfrak{a}}_1(0)+\ldots + \bar{\mathfrak{a}}_n(0)$ is semisimple.
\end{theorem} We call the matrix \(\overline A\) of Theorem \ref{linearnormalform} the {\it linear almost normal form} of the linearization matrix $A$. When $\mathfrak{a}_0(0)$ and $\mathfrak{a}_0(0)+\bar{\mathfrak{a}}_1(0)+\ldots +\bar{\mathfrak{a}}_n(0)$ are both semisimple, then the desired $SN$-decomposition of the linearization can be read off from this almost normal form. 


We shall prove Theorem \ref{linearnormalform} in Section \ref{semigroupringsection} below. But before we do so, we would like to provide an alternative proof in case $n=2$ here: in this case the theorem follows quite easily from an explicit matrix computation.  The proof in Section \ref{semigroupringsection} will be a bit more abstract.

\begin{proofof}{Theorem \ref{linearnormalform} in case ${\bf n=2}$}\label{ex3}
We shall put $\lambda=0$ and write $\mathfrak{a}_i=\mathfrak{a}_i(0)$. Then, in case $n=2$, the matrix (\ref{matrixagain}) takes the form
\begin{align}\label{matrix3} 
A=\left( \begin{array}{rrr} \mathfrak{a}_0 + \mathfrak{a}_1+ \mathfrak{a}_2 & 0 & 0 \\ \mathfrak{a}_1+ \mathfrak{a}_2 & \mathfrak{a}_0 & 0 \\ \mathfrak{a}_2 & \mathfrak{a}_1 & \mathfrak{a}_0 \end{array}\right) 
\, 
 \end{align}
for certain $\mathfrak{a}_0, \mathfrak{a}_1, \mathfrak{a}_2\in \lieg$. We can decompose $A$ as $A=A_S+A_N$ with
\begin{align}\label{decomp3} 
A_S:= \left( \begin{array}{rrr} \mathfrak{a}_0 + \mathfrak{a}_1+ \mathfrak{a}_2 & 0 & 0 \\ \mathfrak{a}_1+ \mathfrak{a}_2 & \mathfrak{a}_0 & 0 \\ \mathfrak{a}_1+\mathfrak{a}_2 & 0 & \mathfrak{a}_0 \end{array} \right) \ \mbox{and}\ A_N:= \left( \begin{array}{rrr} 0 & 0 & 0 \\  0 & 0 & 0 \\ -\mathfrak{a}_1 & \mathfrak{a}_1 & 0 \end{array} \right) \, .
 \end{align}
It is clear that $A_N$ is nilpotent. In addition, we can think of $A_S$ as semisimple, because
$$({\rm id}_V, {\rm id}_V, {\rm id}_V), (0, {\rm id}_V, 0) \ \mbox{and}\ (0,0,{\rm id}_V)$$ 
are
``eigenvectors'' of $A_S$ with respectively the ``eigenvalues'' 
$$\mathfrak{a}_0+\mathfrak{a}_1+\mathfrak{a}_2, \mathfrak{a}_0\ \mbox{and}\ \mathfrak{a}_0\, .$$ 
Of course, these eigenvalues are actually linear maps, namely elements of $\lieg$, and it is clear that $A_S$ is truly semisimple only when these eigenvalues are semisimple elements of $\lieg$. 

The decomposition $A=A_S+A_N$ may of course not be the $SN$-decomposition of $A$, because $A_S$ and $A_N$ in general need not commute. In fact, one computes that
\begin{align}
\nonumber
\left[ \left( \begin{array}{rrr} \mathfrak{a}_0 + \mathfrak{a}_1+ \mathfrak{a}_2 & 0 & 0 \\ \mathfrak{a}_1+ \mathfrak{a}_2 & \mathfrak{a}_0 & 0 \\ \mathfrak{a}_1+\mathfrak{a}_2 & 0 & \mathfrak{a}_0 \end{array} \right), \left( \begin{array}{rrr} 0 & 0 & 0 \\  0 & 0 & 0 \\ -\mathfrak{a}_1 & \mathfrak{a}_1 & 0 \end{array} \right) \right]
=\left( \begin{array}{rrr} 0 & 0 & 0 \\ 0 & 0 & 0 \\ -[\mathfrak{a}_0, \mathfrak{a}_1] & [\mathfrak{a}_0, \mathfrak{a}_1] & 0 
\end{array}\right)\, .
\end{align}
We shall resolve this problem by transforming $A$ to a matrix $\overline A$ for which $[\mathfrak{a}_0, \bar{\mathfrak{a}}_1]=0$. 
This works as follows. First, we define, for some $\mathfrak{b}_1\in \lieg$, the function 
$$g_1:V^3\to V\ \mbox{by}\ g_1(X_0, X_1, X_2):= \mathfrak{b}_1 (X_1-X_2)\, .$$ 
Then $\gamma_{g_1}:V^3\to V^3$ is a linear map of the form
$$\gamma_{g_1} = 
\left( \begin{array}{ccc} 0 & 0 & 0 \\ 0 & 0 & 0 \\ -\mathfrak{b}_1 & \mathfrak{b}_1 & 0 \end{array}\right)\, . $$
Moreover, it holds that 
$$e^{\gamma_{g_1}} = \exp \left( \begin{array}{ccc} 0 & 0 & 0 \\ 0 & 0 & 0 \\ -\mathfrak{b}_1 & \mathfrak{b}_1 & 0 \end{array}\right) =  
\left( \begin{array}{ccc} 1 & 0 & 0 \\ 0 & 1 & 0 \\ -\mathfrak{b}_1 & \mathfrak{b}_1 & 1 \end{array}\right)\, .$$
A little computation now shows that conjugation with $e^{\gamma_{g_1}}$ transforms $A$ into 
$$
\overline A \!=\! e^{\gamma_{g_1}} \circ A \circ e^{-\gamma_{g_1}} \!=\! \left( \begin{array}{rrr} \mathfrak{a}_0 + \mathfrak{a}_1+ \mathfrak{a}_2 & 0 & 0 \\ \mathfrak{a}_1 + \mathfrak{a}_2 & 
\mathfrak{a}_0 & 0 \\ \mathfrak{a}_2+[\mathfrak{a}_0,\mathfrak{b}_1] & \mathfrak{a}_1-[\mathfrak{a}_0,\mathfrak{b}_1]& \mathfrak{a}_0 \end{array}\right) \! =\!
\left( \begin{array}{rrr} \mathfrak{a}_0 + \bar{\mathfrak{a}}_1+\bar{\mathfrak{a}}_2& 0 & 0 \\ \bar{\mathfrak{a}}_1+\bar{\mathfrak{a}}_2 & \mathfrak{a}_0 & 0 \\ 
\bar{\mathfrak{a}}_2 & \bar{\mathfrak{a}}_1 & \mathfrak{a}_0 \end{array}\right) 
\, .
$$
Here, we defined 
$$\bar{\mathfrak{a}}_1:= \mathfrak{a}_1-[\mathfrak{a}_0,\mathfrak{b}_1]\ \mbox{and}\ \bar{\mathfrak{a}}_2:= \mathfrak{a}_2+[\mathfrak{a}_0,\mathfrak{b}_1] \, .$$
The essential step is now to choose \(\mathfrak{b}_1\) in such a way that $\bar{\mathfrak{a}}_1$ commutes with \(\mathfrak{a}_0\). 
This is possible because we assumed that \(\mathfrak{a}_0\) is semisimple in $\lieg$, so that  
$$\lieg=\im \ad_{\mathfrak{a}_0}\oplus \ker \ad_{\mathfrak{a}_0}\ \, \mbox{in which}\ \, \ad_{\mathfrak{a}_0}: \mathfrak{b}\mapsto [\mathfrak{a}_0, \mathfrak{b}], \, \lieg \to \lieg\, .$$
Hence, we can decompose $\mathfrak{a}_1=\mathfrak{a}_1^{\rm im} + \mathfrak{a}_1^{\rm ker}$ for unique $\mathfrak{a}_1^{\rm im}\in \im \ad_{\mathfrak{a}_0}$ and $\mathfrak{a}_1^{\rm ker}\in \ker \ad_{\mathfrak{a}_0}$. If we now choose $\mathfrak{b}_1$ so that $\ad_{\mathfrak{a}_0}(\mathfrak{b}_1)=\mathfrak{a}_1^{\rm im}$, then we obtain as a result that $\bar{\mathfrak{a}}_1 = \mathfrak{a}_1-[\mathfrak{a}_0, \mathfrak{b}_1] = \mathfrak{a}_1-\ad_{\mathfrak{a}_0}(\mathfrak{b}_1)=\mathfrak{a}_1 - \mathfrak{a}_1^{\rm im} = \mathfrak{a}_1^{\ker}$, and hence that 
$$[\mathfrak{a}_0, \bar{\mathfrak{a}}_1] = [\mathfrak{a}_0, \mathfrak{a}_1^{\rm ker}]=\ad_{\mathfrak{a}_0}(\mathfrak{a}_1^{\rm ker})=0\, .$$ 
With such choice of $\mathfrak{b}_1$, we obtain that $\overline A_S$ and $\overline A_N$ commute as required.  
\end{proofof}

\begin{remark}
In the process of normalizing $\mathfrak{a}_1$ into $\bar{\mathfrak{a}}_1=\mathfrak{a}_1-[\mathfrak{a}_0, \mathfrak{b}_1]$, we automatically change $\mathfrak{a}_2$ into $\bar{\mathfrak{a}}_2=\mathfrak{a}_2+[\mathfrak{a}_0, \mathfrak{b}_1]$. This means that $\bar{\mathfrak{a}}_2$ will in general not be zero if $\mathfrak{a}_2=0$. Thus, already when we put the linear part of equations (\ref{network3}) in normal form, we obtain a system that is not of the restricted form (\ref{network3}) but of the form (\ref{network}). 
\end{remark}
\noindent Our proof of Theorem \ref{linearnormalform} for general $n$ is similar to the above matrix computation for $n=2$. One could give this proof using the same matrix notation, but we found it more convenient to introduce a more abstract setting first.

\section{A semigroup ring}\label{semigroupringsection}
In our proof of Theorem \ref{linearnormalform} we shall make use of formal expressions of the form 
$$\sum_{i=0}^n \mathfrak{a}_i \sigma_i \ \mbox{with}\ \mathfrak{a}_i \in \lieg\ \mbox{and} \ \sigma_i\in \Sigma\, .$$
We shall denote the space of such expressions as
$$M=\lieg \Sigma = \left\{ \sum_{i=0}^{n} \mathfrak{a}_i \sigma_i \, | \, \mathfrak{a}_i\in \lieg, \sigma_i\in \Sigma \right\}\, .$$
Moreover, when 
 $$A=\sum_{i=0}^n\mathfrak{a}_i\sigma_i\ \mbox{and}\ B=\sum_{j=0}^n\mathfrak{b}_j\sigma_j $$
are two elements of $M$, then we define their product in $M$ as  
\begin{align}\label{Mproduct}
 A \cdot B :=& \sum_{i=0}^n\sum_{j=0}^n (\mathfrak{a}_i \mathfrak{b}_j)(\sigma_i \sigma_j)  = \sum_{i=0}^n\sum_{j=0}^n (\mathfrak{a}_i \mathfrak{b}_j)\sigma_{{\rm min}\{i+j,n\}}\, .
 \end{align}
The collection $M$ is an example of a semigroup ring, cf. \cite[Example 1.4]{MR1838439}. It can be vieved as a module over the ring $\lieg$ with basis $\Sigma$, and also as a representation of $\Sigma$. For us it will suffice to think of $M$ as an associative algebra that inherits its multiplication from $\lieg$ and $\Sigma$. 

One can remark that this algebra is graded: if we define for $0\leq k \leq n$
 the collection
$$M_k: = \lieg \{\sigma_k, \ldots, \sigma_n\} = \left\{ \sum_{i=k}^{n} \mathfrak{a}_i \sigma_i \, | \, \mathfrak{a}_i\in \lieg\right\} \subset M\, ,$$
then $M_0=M$ and for any $0\leq k,l \leq n$ it holds by (\ref{Mproduct}) that 
$$M_k \cdot M_l\subset M_{\min\{k+l,n\}}\, .$$ 
Thus, each $M_k$
is an ideal in \(M\) and
$$M=M_0\supset M_1\supset\cdots \supset M_{n}\supset 0$$ is a filtration.

$M$ also has the structure of a Lie algebra. Using that $\Sigma$ is commutative, we find that the Lie bracket of two elements of $M$ is given by the relatively simple expression
\[
[A, B]_M:=A\cdot B-B\cdot A=  \sum_{i=0}^{n}\sum_{j=0}^{n}[\mathfrak{a}_i , \mathfrak{b}_j ] \sigma_{{\rm min}\{i+j, n\}}.
\]
The role of the semigroup ring $M$ is explained in the following proposition: 
\begin{proposition}\label{homomorphismprop}
The assignment 
\begin{align}\label{matrixagainagain}
\sum_{i=0}^n \mathfrak{a}_i\sigma_i\mapsto \left( \begin{array}{rcccc} 
\mathfrak{a}_0 + \mathfrak{a}_1+\ldots + \mathfrak{a}_n & 0 & \cdots & 0& 0\\
\mathfrak{a}_1 + \ldots + \mathfrak{a}_n & \mathfrak{a}_0 &\cdots & 0& 0\\
\vdots\   & \vdots & \ddots & \vdots &  \vdots \\
\mathfrak{a}_{n-1}+\mathfrak{a}_n & \mathfrak{a}_{n-2} & \cdots  & \mathfrak{a}_0& 0\\
\mathfrak{a}_n & \mathfrak{a}_{n-1} & \cdots & \mathfrak{a}_1 & \mathfrak{a}_0
\end{array} \right)\, .
\end{align}
is a homomorphism of associative algebras from $M$ to $\mathfrak{gl}(V^{n+1})$.
\end{proposition}
\begin{proof}
For $0\leq i \leq n$, let us define the $(n+1)\times(n+1)$ matrix
\begin{align} 
\mathfrak{a}_iN_i:= \left( \begin{array}{ccccccc} 
\mathfrak{a}_i & 0 & \cdots & 0& 0& \cdot & 0\\
\mathfrak{a}_i & 0 &\cdots & 0& 0 & \cdots & 0\\
\vdots\   & \vdots & \ddots & \vdots &  \vdots & \ddots & \vdots \\
\mathfrak{a}_i  & 0 & \cdots  & 0& 0& \cdots & 0\\
0 & \mathfrak{a}_i & \cdots & 0 & 0& \cdots & 0 \\
\vdots & \vdots & \ddots & \vdots & \vdots & \ddots & \vdots \\
0 & \vdots & 0& \mathfrak{a}_i &  0 & \cdots & 0
\end{array} \right) \, . \nonumber \\ \nonumber 
\underbrace{\ \ \ \ \ \ \  \ \ \ \ \ \ \ \ \ }\ \ \ \ \
\\
i \ \mbox{zeros}\ \ \ \ \ \ \ \ \  \nonumber 
\end{align}
With this definition, $\mathfrak{a}_i N_i$ is the matrix of the linear network vector field $\gamma\left(\mathfrak{a}_i X_i\right)$, where by $\mathfrak{a}_i X_i: V^{n+1}\to V$ we denote the map $X\mapsto \mathfrak{a}_i X_i$. 

Moreover, the assignment of the proposition is given by
$$M\ni \sum_{i=0}^n \mathfrak{a}_i\sigma_i\mapsto \sum_{i=0}^n \mathfrak{a}_i N_i \in \mathfrak{gl}(V^{n+1})\, .$$
It is easy to compute that 
$$(\mathfrak{a}_iN_{i})(\mathfrak{b}_jN_j)= (\mathfrak{a}_i\mathfrak{b}_j)N_{{\rm min}\{i+j, n\}}\, .$$
By distributivity of the matrix product, it thus follows that 
$$\left(\sum_{i=0}^n \mathfrak{a}_i N_i\right)\left(\sum_{j=0}^n \mathfrak{b}_j N_j\right) = \sum_{i=0}^n\sum_{j=0}^n (\mathfrak{a}_i \mathfrak{b}_j) N_{{\rm min}\{i+j, n\}}\, .$$
This product is homomorphic to the product in the semigroup ring (\ref{Mproduct}).
\end{proof}
It clearly has notational advantages to represent matrices of the form (\ref{matrixagain}) by elements of the semigroup ring, so this is what we will do in the remainder of this section. We also choose to perform the matrix computations that are necessary for the proof of Theorem \ref{linearnormalform} inside the semigroup ring $M$ and not inside $\mathfrak{gl}(V^{n+1})$. 

But we stress that Proposition \ref{homomorphismprop} proves it justified to think of $\sum_{i=0}^n \mathfrak{a}_i\sigma_i\in M$ as the $(n+1)\times(n+1)$-matrix $\sum_{i=0}^n \mathfrak{a}_iN_i$, which in turn is the matrix of the map
$$\gamma\left(\sum_{i=0}^n \mathfrak{a}_i X_i\right): V^{n+1}\to V^{n+1}\, .$$ 

\begin{proofof}{Theorem \ref{linearnormalform}}
Let us define for $0\leq i \leq n$, 
$$\mu_i:=\sigma_i-\sigma_{n}\in M_i\, .$$
Then it holds that \(\mu_n=0\) and that
$$\mu_i \mu_j = (\sigma_i-\sigma_n)(\sigma_j-\sigma_n)  = \sigma_{{\rm min}\{j+i, n\}} - \sigma_n =\mu_{{\rm min}\{i+j,n\}}\, .$$ 
In particular, every $\mu_i$ is nilpotent in $M$. We also remark that the collection 
$$\mu_0, \ldots, \mu_{n-1}, \sigma_{n}$$
is a $\lieg$-basis for the module $M$.

Let us now assume that for some $1\leq k\leq n-1$ there exist elements of the form 
$$G_1=\mathfrak{b}_1\mu_1, \ldots, G_{k-1} = \mathfrak{b}_{k-1}\mu_{k-1} \in M$$ such that
$$\tilde A := \left(e^{G_{k-1}}\cdots e^{G_1}\right)\cdot A \cdot \left(e^{-G_1}\cdots e^{-G_{k-1}}\right) =  \sum_{i=0}^{n} \tilde{\mathfrak{a}}_i \sigma_i $$
has the property that $\tilde{\mathfrak{a}}_0=\mathfrak{a}_0$ and $\ad_{\mathfrak{a}_0}(\tilde{\mathfrak{a}}_i)=[\mathfrak{a}_0, \tilde{\mathfrak{a}}_i]=0$ for all $1\leq i\leq k-1$.


Then we pick $\mathfrak{b}_k\in \lieg$ arbitrary and define $$G_k:=\mathfrak{b}_k\mu_k\, .$$
Clearly, $G_k$ is nilpotent in $M$ and hence \(e^{G_k}\) is a finite series expansion and easy to compute.
One finds that
\begin{align}\nonumber
 e^{G_k}\cdot \tilde A\cdot e^{-G_k}&=e^{\mathrm{ad}^M_{G_k}}(A) \\ \nonumber
&= \sum_{j=0}^\infty\frac{1}{j!}\left(\mathrm{ad}^M_{G_k}\right)^j(\tilde A)\\ \nonumber 
&= \sum_{j=0}^\infty\frac{1}{j!}\sum_{i=0}^n \ad_{\mathfrak{b}_k}^j(\tilde{\mathfrak{a}}_i)\mu_k^j\sigma_i\\ \nonumber
&= \sum_{i=0}^k\tilde{\mathfrak{a}}_i\sigma_i -  [\mathfrak{a}_0, \mathfrak{b}_k]\sigma_k \!\!\!\! \mod M_{k+1} . \nonumber 
\end{align}
As was explained in Section \ref{sec:0}, we can now choose $\mathfrak{b}_k$ in such a way that $\bar{\mathfrak{a}}_k:= \tilde{\mathfrak{a}}_k  -  [\mathfrak{a}_0, \mathfrak{b}_k]$ commutes with $\mathfrak{a}_0$, because  $\mathfrak{a}_0$ is semisimple. 

This proves, by induction, that $A=\sum_{j=0}^n\mathfrak{a}_j\sigma_j$ can be normalized into 
$$\overline A=\mathfrak{a}_0\sigma_0+\sum_{i=1}^n\bar{\mathfrak{a}}_i\sigma_i \ \mbox{such that}\ [\mathfrak{a}_0, \bar{\mathfrak{a}}_i]=0 \ \mbox{for}\ 1\leq i\leq n-1\, .$$
Next, let us decompose $\overline A = \mathfrak{a}_0\sigma_0 + \sum_{i=1}^n \bar{\mathfrak{a}}_i\sigma_i$ as $\overline A=\overline A_S + \overline A_N$ with  
$$\overline A_S=\mathfrak{a}_0 \sigma_0 +\left(\sum_{i=1}^{n} \bar{\mathfrak{a}}_i \right) \sigma_{n} \ \mbox{and}\  
\overline A_N=\sum_{i=1}^{n-1} \bar{\mathfrak{a}}_i \left( \sigma_i - \sigma_n\right) = \sum_{i=1}^{n-1} \bar{\mathfrak{a}}_i\mu_i  \, .$$
Then, because $\sigma_n\mu_i=0$, it is clear that \([\overline A_S, \overline A_N]_M=0\), that \(\overline A_N\) is nilpotent, and that
$$\overline A_S \mu_i=\mathfrak{a}_0 \mu_i \ \mbox{for}\ i=0,\cdots,n-1\ \mbox{and}\ \overline A_S\sigma_{n}=
\left(\mathfrak{a}_0+\sum_{i=1}^{n-1} \bar{\mathfrak{a}}_i\right) \sigma_{n}\, .$$ 
This shows that $\mu_0, \ldots, \mu_{n-1}$ and \(\sigma_{n}\) form a basis of ``eigenvectors'' of \(\overline A_S\) with respectively the ``eigenvalues'' $\mathfrak{a}_0, \cdots,\mathfrak{a}_0$ and $\mathfrak{a}_0+\sum_{i=1}^{n} \bar{\mathfrak{a}}_i$.
In particular, \(\overline A_S\) is semisimple if these eigenvalues are semisimple in $\lieg$.
\end{proofof}
\begin{remark}
If one wishes, one can also bring \(\mathfrak{a}_0\) in Jordan normal form.
Indeed, if \(\mathfrak{b}_0\in \lieg\) is an invertible map for which \(\mathfrak{b}_0\mathfrak{a}_0 \mathfrak{b}_0^{-1}\) is in Jordan normal form with respect to some basis of $V$, then one can
define $G_0=\mathfrak{b}_0\sigma_0$. We then observe that $G_0$ is invertible in $M$, with $G_0^{-1}=\mathfrak{b}_0^{-1}\sigma_0$, and that $$G_0\cdot A \cdot G_0^{-1} =
\sum_{i=0}^{n} (\mathfrak{b}_0 \mathfrak{a}_i \mathfrak{b}_0^{-1}) \ \sigma_i \, .$$ 
This shows that one may assume that the matrix of $\mathfrak{a}_0$ is given in Jordan normal form.
\end{remark}

\begin{remark} 
For the proof of Theorem \ref{linearnormalform}, the assumption that \(\mathfrak{a}_0\) is semisimple is essential. If $\mathfrak{a}_0$ is not semisimple, then it can in general not be arranged that $[\mathfrak{a}_0, \bar{\mathfrak{a}}_i]=0$ and hence it is not clear that \([\overline A_S, \overline A_N]_M=0\) either.
\end{remark}

\begin{remark}
We stress that for $\overline A_S$ to be truly semisimple, it is necessary that its eigenvalues $\mathfrak{a}_0, \mathfrak{a}_0+\bar{\mathfrak{a}}_1+\ldots + \bar{\mathfrak{a}}_n$ are semisimple in $\lieg$. Nevertheless, it will become clear that this is not important for us.
\end{remark}
\begin{remark}
One could also try to normalize \(\mathfrak{a}_{n}\) by conjugating with an element of the form $e^{G_n}$, with $G_n=\mathfrak{b}_n\sigma_n$ for some $\mathfrak{b}_n\in \lieg$. This indeed transforms $\bar{\mathfrak{a}}_n$ further, but we remark that this normalization is:
\begin{itemize}
\item[1.]
Much more difficult to carry out than the normalization of the $\mathfrak{a}_i$ with $1\leq i\leq n-1$, because $\sigma_n$ is not nilpotent but idempotent.
\item[2.]
Not necessary for the computation of the \(SN\)-decomposition of $A$. 
Recall that we are interested in this $SN$-decomposition for the computation of a nonlinear
normal form. 
\end{itemize}
For us, the equivalent of the Jordan normal form of \(A=\sum_{i=0}^n\mathfrak{a}_i\sigma_i\) will therefore be a matrix of the form
\(\overline A=\mathfrak{a}_0\sigma_0 +\sum_{i=1}^{n} \bar{\mathfrak{a}}_i \sigma_i$ with $[\mathfrak{a}_0, \bar{\mathfrak{a}}_i]=0$ for $1\leq i\leq n-1$. Since we do not normalize \(\bar{\mathfrak{a}}_{n}\), we call $\overline A$ an {\em almost normal form} of \(A\).
\end{remark}

\section{Analysis of the nonlinear normal form}\label{sec:4}

In this section, we prove Theorem \ref{mainthm} of the introduction, that is formulated more precisely as Theorem \ref{mainthm2} below. This theorem generalizes the results in \cite{elmhirst}, \cite{claire} and \cite{golstew} to feed-forward chains of arbitrary length. We like to point out that, except for the linear algebra, various computational aspects of our proof are the same as those given in \cite{golstew}.

Before we formulate Theorem \ref{mainthm2}, we remark that when $V=\R^2$ and $\mathfrak{a}_0(0): V\to V$ has nonzero eigenvalues $\pm i\omega_0$, then we may assume that $\mathfrak{a}_0(0)$ is given in Jordan normal form
$$\mathfrak{a}_0(0)=\left(\begin{array}{cc} 0 & -\omega_0\\ \omega_0 & 0\end{array}\right)\, .$$
As a consequence, it is then convenient to identify $X= (X^1,X^2)\in \R^2$ with the complex number $X^1+iX^2\in \C$. This turns the linear map $\mathfrak{a}_0(0):\R^2\to\R^2$ into the multiplication by $i\omega_0$ from $\C$ to $\C$ and allows for Theorem \ref{mainthm2} to be formulated conveniently as follows:

\begin{theorem}\label{mainthm2}
Let $V=\R^2\cong \C$ and $f:V^{n+1}\times \R \to V$ a smooth function. Assume that $f(0;\lambda)=0$ and that $\mathfrak{a}_0(0)$ has eigenvalues $\pm i\omega_0\neq 0$. Furthermore, assume the following generic conditions on $f(X; \lambda)$:
\begin{itemize}
\item[1.] Persistence of the steady state: $\mathfrak{a}_0(0)+\ldots +\mathfrak{a}_n(0)$ is invertible.
\item[2.] Eigenvalue crossing: $\left. \frac{d}{d\lambda}\right|_{\lambda=0}{\rm tr}\ \mathfrak{a}_0(\lambda)\neq 0$.
\item[3.] Nilpotency: ${\rm tr}\, \mathfrak{a}_1(0) \neq 0$.
\item[4.] Nonlinearity: $\frac{\partial^3 ({\rm Re}\, \overline f)(0;0)}{\p ({\rm Re}\, X_0)^3}\neq 0$.
\end{itemize}
Then the local normal form of equations (\ref{network}) near $(x; \lambda)=(0;0)$ supports $n$ branches of periodic solutions (counted excluding the trivial steady state $x(t)\equiv 0$). They are given for $r=1, \ldots, n$ and for certain $B_1(\lambda), \ldots, B_n(\lambda)\in \C$ by
$$x_0(t)= \ldots = x_{r-1}(t)\equiv 0, x_{r}(t) = B_1(\lambda) e^{i\omega(\lambda) t}, \ldots, x_n(t)= B_{n-r+1}(\lambda) e^{i\omega(\lambda)t}\, .$$  
Frequency and amplitudes satisfy \(\omega(\lambda) = \omega_0+\mathcal{O}(|\lambda|) \) and
\( |B_i(\lambda)| \sim |\lambda|^{\kappa_i}\) with \(\kappa_i:=\frac{1}{2}\frac{1}{ 3^{i-1}}\) for $i=1,\ldots,n$. These branches are hyperbolic if and only if $\mathfrak{a}_0(0)+\ldots+\mathfrak{a}_n(0)$ is hyperbolic. 
\end{theorem}
Before we prove Theorem \ref{mainthm2}, let us summarize the results of the previous sections as follows:
\begin{proposition} Under the assumptions of Theorem \ref{mainthm2}, the vector field $\gamma_f$ admits a local normal form $\gamma_{\overline f}$ near $(x;\lambda)=(0;0)$ for which it holds that 
\begin{align}\label{invariancecomplex}
\overline f(e^{i\omega_0s}X_0, \ldots, e^{i\omega_0s}X_{n-1}, 0; \lambda) = e^{i\omega_0s}\overline f(X_0, \ldots, X_{n-1}, 0; \lambda)\ \mbox{for all}\ s\in \R \, .
\end{align}
\end{proposition}
\begin{proof}
By Theorem \ref{linearnormalform} there exist linear maps $g_i:V^{n+1}\to V$ ($1\leq i \leq n-1$) of the form 
$$g_i(X)=\mathfrak{b}_i(X_i-X_n) \ \mbox{so that}\  e^{\gamma_{g_{n-1}}} \circ \ldots \circ e^{\gamma_{g_{1}}}: V^{n+1}\to V^{n+1}$$ 
transforms $A=D_x\gamma_f(0;0) = \gamma_{f_{0,0}}$ into 
$$\overline A:= \left( e^{\gamma_{g_{n-1}}} \circ \ldots \circ e^{\gamma_{g_{1}}} \right)\circ A\circ \left( e^{-\gamma_{g_{1}}} \circ \ldots \circ e^{-\gamma_{g_{n-1}}} \right) \ \mbox{in ``almost normal form''.}$$ 
This means that $\overline A$ decomposes as $\overline A=\overline A_S+\overline A_N$ with $\overline A_N$ nilpotent, $[\overline A_S, \overline A_N]=0$ and
$$ \overline A_S (x) =  (0, \mathfrak{a}_0(0)x_1, \ldots, \mathfrak{a}_0(0)x_n)  + \mathcal{O}(x_0)  \cong (0, i\omega_0 x_1, \ldots, i\omega_0 x_n) + \mathcal{O}(x_0)  \, .$$
In particular, we find that the subspace 
$$V^{n+1}_0:=\{x\in V^{n+1}\, |\, x_0=0\}$$ is invariant under both $\overline A_S$ and $\overline A_N$. 

It is clear that the restriction of $\overline A_S$ to $V_0^{n+1}$ is semisimple. Because $\overline A_N$ is nilpotent, this proves that $\overline A = \overline A_S+\overline A_N$ is the $SN$-decomposition of the restriction of $\overline A$ to $V^{n+1}_0$. 

Because the transformation $e^{\gamma_{g_{n-1}}} \circ \ldots \circ e^{\gamma_{g_{1}}}$ is a composition of flows of semigroup network vector fields, it transforms $\gamma_f$ into another coupled cell network
\begin{align}\nonumber
& (e^{\gamma_{g_{n-1}}} \circ \ldots \circ e^{\gamma_{g_{1}}} )_*\gamma_f  =  \gamma_{\widetilde f}\, .
 \end{align}
It obviously holds that 
$$D_x\gamma_{\widetilde f}(0;0)=\gamma_{\widetilde f_{0,0}}=\overline A\, .$$ 
By Theorem \ref{normalformtheoremparameters}, it can now be arranged that the local normal form of $\gamma_{\widetilde f}$ near $(x;\lambda)=(0;0)$ is a network $\gamma_{\overline f}$ that commutes with the flow of $\overline A_S$, i.e.
$$\gamma_{\overline f}(e^{s\overline A_S}x) = e^{s\overline A_S}(\gamma_{\overline f}(x))\ \mbox{for all} \ s\in \R \, .$$
Restricted to $V_0^{n+1}$, this means that
$$\overline f(e^{i\omega_0s}x_j, \ldots, e^{i\omega_0s}x_{1}, 0, \ldots,0;\lambda) = e^{i\omega_0s}\overline f(x_j, \ldots, x_1, 0, \ldots, 0; \lambda)\  \mbox{for all}\ 1\leq j \leq n\, .$$
The latter is true if and only if (\ref{invariancecomplex}) holds.
\end{proof}

\begin{remark}
We remark that on the invariant subspace $V_0^{n+1}$ the normal form symmetry of $\gamma_{\overline f}$
reduces to the classical normal form symmetry $$(0, x_1, \ldots, x_n)\mapsto (0, e^{i\omega_0s}x_1, \ldots,  e^{i\omega_0s}x_n)$$ 
of the $1:1:\ldots:1$-resonant harmonic oscillator. One could also try to normalize $\gamma_{\overline f}$ further with respect to the nilpotent operator $\overline A_N$, see \cite{sanvermur}. We will not exploit this further freedom though.
\end{remark}
We are now ready to prove Theorem \ref{mainthm2}, using computations similar to those given in \cite{golstew}:

\begin{proofof}{Theorem \ref{mainthm2}} 
In this proof, we shall use the normal form symmetry (\ref{invariancecomplex}). In fact, applied to a monomial 
$$\overline f(X_0, \ldots, X_{n-1},0; \lambda)=X_0^{\beta_0}\cdots X_{n-1}^{\beta_{n-1}}\cdot\overline{X}_0^{\gamma_0}\cdots \overline{X}_{n-1}^{\gamma_{n-1}}\cdot \lambda^{\delta}$$
 equation (\ref{invariancecomplex}) yields the restriction that $\sum_{j=0}^{n-1} (\beta_j -\gamma_j)=1$. Hence, the general polynomial or smooth $\overline f(X;\lambda)$ that satisfies (\ref{invariancecomplex}) must be of the form
$$\overline f(X_0, \ldots, X_{n-1}, 0; \lambda) = \sum_{j=0}^{n-1} F_j(\left. X_k{\overline X}_l \right|_{k,l=0, \ldots, n-1}; \lambda)X_j$$
for certain polynomial or smooth complex-valued functions $F_j$ of the parameter and of the complex invariants 
$$X_k{\overline X}_l = (X_k^1X_l^1+X_k^2X_l^2) + i(X_k^2X_l^1-X_k^1X_l^2)\ \mbox{with}\ 0\leq k,l\leq n-1\, .$$
Moreover, the assumptions of the theorem imply that the first two of these functions Taylor expand as follows:
\begin{align}
F_0(\ldots, X_k{\overline X}_l, \ldots; \lambda) &= i \omega_0+\alpha \lambda + C |X_0|^2 \nonumber  \\ 
& + \mathcal{O}\left(|X_0|^4 + |\lambda|\cdot |X_0|^2 + |\lambda|^2 + ||(X_1, \ldots, X_n)||^2 \right) \, , \nonumber \\ \nonumber 
F_1(\ldots, X_k{\overline X}_l, \ldots;\lambda) & = \beta +\mathcal{O}(||(X_0, \ldots, X_n)||^2 + |\lambda|)\, . 
\end{align}
Here, we defined $C:=\frac{1}{3}\frac{\partial^3{\overline f}(0;0)}{\partial ({\rm Re}\, X_0)^3}\in \C$, for which it holds by assumption that ${\rm Re}\, C \neq 0$. Moreover, $\alpha=\alpha_1+i\alpha_2\in \C$ and $\beta=\beta_1+i\beta_2\in \C$
are those numbers for which
\begin{align}
\bar{\mathfrak{a}}_0(\lambda)&=\left(\begin{array}{rr} \alpha_1 \lambda + \mathcal{O}(|\lambda|^2)  & -\omega_0 -\alpha_2\lambda+ \mathcal{O}(|\lambda|^2) \\ \omega_0 + \alpha_2\lambda + \mathcal{O}(|\lambda|^2) & \alpha_1 \lambda +\mathcal{O}(|\lambda|^2)\end{array}\right) \ \mbox{and} \nonumber \\ \nonumber
 \bar{\mathfrak{a}}_1(\lambda)&=\left(\begin{array}{rr} \beta_1+ \mathcal{O}(|\lambda|) & - \beta_2+ \mathcal{O}(|\lambda|)\\ \beta_2+ \mathcal{O}(|\lambda|) &\beta_1+ \mathcal{O}(|\lambda|) \end{array}\right)\, 
\, .
\end{align}
In particular it holds that 
\begin{align}
\alpha_1 & = \frac{1}{2}\left. \frac{d}{d\lambda}\right|_{\lambda=0} \!\!  {\rm tr} \, \bar{\mathfrak{a}}_0(\lambda) = \frac{1}{2}\left. \frac{d}{d\lambda}\right|_{\lambda=0} \!\! {\rm tr} \, \mathfrak{a}_0(\lambda)\neq 0 \ \mbox{and that} \nonumber \\ \nonumber 
\beta_1 & = \frac{1}{2}{\rm tr}\, \bar{\mathfrak{a}}_1(0)  = \frac{1}{2}{\rm tr}\, \mathfrak{a}_1(0) \neq 0\, .
\end{align} 

\noindent Here, we have used that the trace of $\bar{\mathfrak{a}}_0(\lambda)=\mathfrak{a}_0(\lambda)-[\mathfrak{a}_0(0), \mathfrak{b}_0(\lambda)]$ is equal to that of $\mathfrak{a}_0(\lambda)$ and the trace of $\bar{\mathfrak{a}}_1(\lambda) = \mathfrak{a}_1(\lambda) - [\mathfrak{a}_0(0), \mathfrak{b}_1(\lambda)]$ is equal to that of $\mathfrak{a}_1(\lambda)$.

Using this information, we will now try to find the relative equilibria of the normal form. To this end, we set $x_0(t)=x_0^*(t):= 0$ and solve for the periodic solutions of (\ref{network}) by solving for $j=1, \ldots, n$ the equations $\dot x_j=f(x_j, x_{j-1}, \ldots, x_1, x_0, \ldots, x_0; \lambda)=  0$ consecutively.

To find $x_1(t)$, we solve  
\begin{align}\nonumber
\frac{dx_1}{dt}& =f(x_1, 0, \ldots, 0; \lambda) = F_0(|x_1|^2, 0;\lambda)x_1 = \\ \label{hopfnormalform} 
&\left( i\omega_0+\alpha \lambda + C|x_1|^2 + \mathcal{O}\left(|x_1|^4+|\lambda|\cdot |x_0|^2 + |\lambda|^2\right)\right) x_1\, .
\end{align}
Equation (\ref{hopfnormalform}) is the normal form of the ordinary Hopf bifurcation. Its relative equilibria are found by making an Ansatz $x_1(t) = B_1 e^{i\omega t}$ for $B_1\in \C$ and $\omega\in \R$. Ignoring the solution $B_1=0$, this Ansatz leads to a complex equation in $|B_1|^2$ and $\omega$:
\begin{align}\label{hopfnormalformim}
i\omega = i\omega_0 + \alpha \lambda + C|B_1|^2 + \mathcal{O}(|B_1|^4 + |\lambda|\cdot|B_1|^2+|\lambda|^2) \, .
\end{align}
The real part of this equation 
\begin{align}
\label{real}
0=\alpha_1\lambda + {\rm Re}\, C |B_1|^2 + \mathcal{O}(|B_1|^4 + |\lambda|\cdot |B_1|^2 + |\lambda|^2)
\end{align}
can only be solved for small values of $\lambda$ and for small but nonnegative values of $|B_1|^2$ if $$\alpha_1\lambda/{\rm Re}\, C < 0\, .$$ 
This criterion determines the sign of $\lambda$ and thus whether the bifurcation is subcritical or supercritical. Thus, from now on we choose the sign of $\lambda$ so that $\alpha_1\lambda/{\rm Re}\, C < 0$. 

Because there is no restriction on the argument of $B_1$, let us try to find $B_1=|B_1|>0$ real. Anticipating that $|B_1|$ will be a smooth function of $|\lambda|^{\frac{1}{2}}$, we define $\mu := |\lambda|^{\frac{1}{2}}$ and $|B_1| =: \mu Z_1$. Then equation (\ref{real}) becomes a smooth equation in $Z_1$ and $\mu$ of the form 
$$\mu^2 \left( |\alpha_1| - |{\rm Re}\, C| Z_1^2\right) + \mathcal{O}(\mu^4)= 0\, .$$
Dividing this equation by $\mu^2$, we see that we need to solve an equation of the form
$$h(Z_1, \mu) =  |\alpha_1| - |{\rm Re}\, C| Z_1^2 + \mathcal{O}(\mu^2) = 0\, $$
for some smooth function $h(Z_1, \mu)$. Let $z_1:=\sqrt{|\alpha_1/\rm Re\, C|}>0$ so that $z_{1}$ is the positive solution to the equation $h(Z_{1}, 0) = 0$. Then clearly $D_{Z_{1}}h(z_{1},0)= -2|{\rm Re}\, C| z_{1}\neq 0$ and thus by the implicit function theorem there exists a smooth function $Z_{1}^*:[0, \mu_0)\to\R_{> 0}$ of the form $Z_1^*(\mu) = z_{1}+\mathcal{O}(\mu)$ that solves $h(Z_1^*(\mu), \mu)=0$. Correspondingly, 
$$B_1(\lambda)=|B_1(\lambda)| = B_1^*(|\lambda|^{\frac{1}{2}}):=|\lambda|^{\frac{1}{2}}Z_1^*(|\lambda|^{\frac{1}{2}}) = \sqrt{\left| \alpha_1 \lambda/{\rm Re}\, C \right|} + \mathcal{O}(|\lambda|) \sim \sqrt{|\lambda|} = |\lambda|^{\kappa_1}\, .
$$
is a branch of solutions to equation (\ref{real}). 

The imaginary part of equation (\ref{hopfnormalformim}) immediately gives that the frequency $\omega(\lambda)$ is a smooth function of $\lambda$, given by
\begin{align}
\label{arg} & \omega(\lambda) = \omega_0 + \alpha_2 \lambda + {\rm Im}\, C|B_1(\lambda)|^2 + \mathcal{O}(|\lambda|^2)\, .
\end{align}
This finishes the analysis of equation (\ref{hopfnormalform}).

Next, assume as induction hypothesis that we have found for $1\leq j\leq n-1$, either for $\lambda \in (-\lambda_0, 0]$ or for $\lambda\in [0,\lambda_0)$, solutions 
$$x_0(t)\equiv 0, x_1(t)=B_1(\lambda)e^{i\omega(\lambda)t}, \ldots, x_{j}(t)=B_{j}(\lambda)e^{i\omega(\lambda)t} $$ 
of the equations $\dot x_0 = f(x_0, \ldots, x_0;\lambda)$, $\ldots$, $\dot x_{j}=f(x_{j}, \ldots, x_0;\lambda)$ so that for all $1\leq i \leq j$, $B_i(\lambda)$ is a smooth function of $|\lambda|^{\kappa_i}$, i.e. $B_i(\lambda)=B_i^*(|\lambda|^{\kappa_i})\sim |\lambda|^{\kappa_i}$ for some smooth function $B_i^*: [0, \lambda_0^{\kappa_i})\to \C$. We already proved this induction hypothesis for $j=1$.

Then by solving $\dot x_{j+1}=f(x_{j+1}, \ldots, x_0;\lambda)$ we shall try to obtain $x_{j+1}(t)$, that is we solve
$$\frac{dx_{j+1}}{dt} = \sum_{i=0}^{n-1}F_i( \ldots; \lambda)x_{{\rm max}\{j+1-i,0\}}  = ( i\omega_0 + \alpha \lambda+ C|x_{j+1}|^2)x_{j+1} + \beta x_{j} + \mathcal{O}(|\lambda|^{\kappa_{j-1}}) \, .$$
Now the Ansatz $x_{j+1}(t)=B_{j+1} e^{i\omega(\lambda)t}$ leads to the equation
\begin{align}\label{Bjeqn}
i\omega(\lambda)B_{j+1} = (i\omega_0 + \alpha \lambda + C|B_{j+1}|^2) B_{j+1} + \beta B_{j} + \mathcal{O}(|\lambda|^{\kappa_{j-1}})\, .
\end{align}
Anticipating that the solution $B_{j+1}$ will have amplitude $|\lambda|^{\kappa_{j+1}}$, let us at this point define the rescaled parameter $\mu:=|\lambda|^{\kappa_{j+1}} = |\lambda|^{\frac{1}{2\cdot 3^{j}}}$ and the rescaled unknown $B_{j+1}=: \mu Z_{j+1}$. Then it holds that $\mu^{3^{j-i+1}} = |\lambda|^{\kappa_{i}}$, which inspires us to define also the (smooth) rescaled functions $Z_1^*, \ldots, Z_j^*$ by
$$B_1^*(|\lambda|^{\kappa_1}) =: \mu^{3^{j}}Z_1^*(\mu^{3^{j}}), \ldots, B_j^*(\lambda^{\kappa_{j}})=:\mu^3Z_j^*(\mu^3)\ \mbox{for} \ \mu=\lambda^{\kappa_{j+1}} \, .$$ 
Then, because $\omega(\lambda)=\omega_0+\mathcal{O}(\lambda)$, equation (\ref{Bjeqn}) obtains the form
$$\mu^3\left( C |Z_{j+1}|^2Z_{j+1} + \beta Z_{j}(0) \right) + \mathcal{O}(\mu^4)=0\, .$$
Dividing by $\mu^3$ we now find that we need to solve an equation of the form 
$$h(Z_{j+1}, \mu)= C|Z_{j+1}|^2Z_{j+1} + \beta Z_{j}(0) + \mathcal{O}(\mu)\, .$$
We solve this equation as follows. First of all, there is a unique $z_{j+1}\in \C$ for which $h(z_{j+1},0)=C|z_{j+1}|^2z_{j+1} + \beta Z_{j}(0) =0$. It clearly holds that $z_{j+1}\neq 0$ because $\beta\neq 0$ and $Z_j(0) \neq 0$ by the induction hypothesis. As a consequence, $$D_{Z_{j+1}}h(z_{j+1},0): v \mapsto 2C|z_{j+1}|^2v+Cz_{j+1}^2{\overline v} \, $$ 
is invertible, since $\det D_{Z_{j+1}}h(z_{j+1},0) = 3|C|^2|z_{j+1}|^4\neq 0$. Thus, by the implicit function theorem, there exists a smooth function $Z_{j+1}^*: [0, \mu_0)\to \C$ so that $Z_{j+1}^*(0)= z_{j+1}\neq 0$ and $h(Z_{j+1}^*(\mu), \mu)=0$. Correspondingly, 
$$B_{j+1}(\lambda)=|B_{j+1}(\lambda)| = B_{j+1}^*(|\lambda|^{\frac{1}{2}}):=|\lambda|^{\kappa_{j+1}}Z_{j+1}^*(|\lambda|^{\kappa_{j+1}})  \sim  |\lambda|^{\kappa_{j+1}}\, 
$$
is a branch of solutions to (\ref{Bjeqn}).
This finishes the induction and proves the existence of the Hopf branch
$$x_0(t)\equiv 0, x_1(t) = B_1(\lambda)e^{i\omega(\lambda)t}, \ldots, x_n(t)= B_n(\lambda)e^{i\omega(\lambda)t} \, .$$
The remaining branches in the statement of the theorem exist by symmetry.

Finally, we consider the linearization of the normal form flow around the periodic solution on the $r$-th branch. Thus, we perturb our relative equilibrium by substituting into the normal form equations of motion
\begin{align}
x_0(t)=\varepsilon y_0(t), & \ldots, x_{r-1}(t)=\varepsilon y_{r-1}(t) \nonumber \\ \nonumber 
x_r(t)=(B_1(\lambda)+\varepsilon y_r(t))e^{i\omega(\lambda)t}, & \ldots, x_n(t)=(B_{n-r+1}(\lambda)+\varepsilon y_n(t)) e^{i\omega(\lambda)t}\, .
\end{align} 
This yields that $\dot y = M(\lambda) y+\mathcal{O}(\varepsilon)$,  the linearization matrix being of lower triangular form
\begin{align}\nonumber
M(\lambda) = \left( \begin{array}{ccccccc} 
\bar{\mathfrak{a}}_0(\lambda)+ \ldots + \bar{\mathfrak{a}}_n(\lambda) & 0 & & \hdots&  & & 0\\
* & \bar{\mathfrak{a}}_0(\lambda) &0 & &  & \hdots & 0\\
   & * & \ddots & &  & & \vdots \\
&  &  & \bar{\mathfrak{a}}_0(\lambda) & 0 & & 0\\
\vdots & \vdots & & * & \mathfrak{b}_1(\lambda) & 0& 0 \\
& & & & & \ddots & \vdots \\
* & * & & \cdots & & * & \mathfrak{b}_{n-r+1}(\lambda)
\end{array} \right)\, .
\end{align}
The stability type of the Hopf curve is thus determined by the maps on the diagonal. The first of these, 
$\bar{\mathfrak{a}}_0(\lambda)+ \ldots+ \bar{\mathfrak{a}}_n(\lambda)=\mathfrak{a}_0(0)+\bar{\mathfrak{a}}_1(0)+ \ldots + \bar{\mathfrak{a}}_n(0)+\mathcal{O}(|\lambda|)$, is hyperbolic because $\mathfrak{a}_0(0)+\ldots + \mathfrak{a}_n(0)$ is hyperbolic and has the same eigenvalues as $\mathfrak{a}_0(0)+ \bar{\mathfrak{a}}_1(0)\ldots + \bar{\mathfrak{a}}_n(0)$. 

The maps $\bar{\mathfrak{a}}_0(\lambda)=i\omega_0 + \alpha \lambda+\mathcal{O}(|\lambda|^2)$ ($r-1$ times) are hyperbolic for $\lambda\neq 0$ because $\alpha_1\neq 0$. Finally, the maps $\mathfrak{b}_1(\lambda), \ldots, \mathfrak{b}_{n-r+1}(\lambda)$ are given asymptotically by 
$$\mathfrak{b}_j(\lambda)v= 2C|B_j|^2v + CB_j^2\overline v + \mathcal{O}(|\lambda|^{\kappa_{j-1}}|v|)\, .$$
These maps are hyperbolic for $\lambda\neq 0$ because $\det \mathfrak{b}_j(\lambda) = 3|C|^2|B_j|^4+\mathcal{O}(|\lambda|^{2\kappa_j\kappa_{j-1}}) > 0$ and ${\rm tr}\, \mathfrak{b}_j(\lambda) = 4({\rm Re}\, C) |B_j|^2+ \mathcal{O}(|\lambda|^{\kappa_{j-1}})\neq 0$. Thus, the branches are hyperbolic. 
\end{proofof}
\begin{remark}
For any one of the Hopf branches given in Theorem \ref{mainthm2} to be stable, it is necessary that ${\rm Re}\, C<0$. In turn, this implies that 
the branches exists for $\lambda$ with $\alpha_1 \lambda>0$. For such $\lambda$ though, the eigenvalues of $\bar{\mathfrak{a}}_0(\lambda) = i\omega_0+\alpha\lambda + \mathcal{O}(|\lambda|^2)$ have positive real part. Thus, the only branch of periodic solutions that can possibly be stable is the branch with the least synchrony and the largest amplitude, i.e. the branch with asymptotics
$$x_0=0, x_1\sim |\lambda|^{\kappa_1}, \ldots, x_n\sim |\lambda|^{\kappa_n}\, .$$
Indeed, this branch is stable precisely when ${\rm Re} \, C<0$ and $\mathfrak{a}_0(0)+\ldots+\mathfrak{a}_n(0)$ only has eigenvalues with negative real parts. 
 \end{remark}
 

\bibliography{CoupledNetworks}
\bibliographystyle{amsplain}
\end{document}